\documentclass[preprint,12pt]{elsarticle}

\usepackage{amssymb,amsmath,amsthm,color}
\usepackage[utf8]{inputenc}
\usepackage{url}
\usepackage{subfigure}

\newcommand{\cD}{\mathcal{D}}

\newcommand{\bsx}{\boldsymbol{x}}

\newcommand{\bszero}{\boldsymbol{0}}

\newcommand{\bst}{\boldsymbol{t}}

\newcommand{\ld}{{\rm ld}\,}
\newcommand{\RR}{\mathbb{R}}
\newcommand{\NN}{\mathbb{N}}

\newcommand{\EE}{\mathbb{E}}
\newcommand{\PP}{\mathbb{P}}
\newcommand{\cP}{\mathcal{P}}

\newtheorem{definition}{Definition}
\newtheorem{theorem}{Theorem}

\newtheorem{corollary}{Corollary}

\newtheorem{lemma}{Lemma}

\journal{Journal of Computational and Applied Mathematics}

\begin{document}

\begin{frontmatter}

%% Title, authors and addresses

%% use the tnoteref command within \title for footnotes;
%% use the tnotetext command for theassociated footnote;
%% use the fnref command within \author or \address for footnotes;
%% use the fntext command for theassociated footnote;
%% use the corref command within \author for corresponding author footnotes;
%% use the cortext command for theassociated footnote;
%% use the ead command for the email address,
%% and the form \ead[url] for the home page:
%% \title{Title\tnoteref{label1}}
%% \tnotetext[label1]{}
%% \author{Name\corref{cor1}\fnref{label2}}
%% \ead{email address}
%% \ead[url]{home page}
%% \fntext[label2]{}
%% \cortext[cor1]{}
%% \address{Address\fnref{label3}}
%% \fntext[label3]{}

\title{Secure pseudorandom bit generators and point sets with low star-discrepancy}

%% use optional labels to link authors explicitly to addresses:
%% \author[label1,label2]{}
%% \address[label1]{}
%% \address[label2]{}
\author[1]{Ana-Isabel G\'{o}mez, Domingo G\'{o}mez-P\'{e}rez\corref{cor1} \fnref{fn1}}
\author[2]{Friedrich Pillichshammer \fnref{fn2}\fnref{gracias}}
\address[1]{Departamento de Matemáticas, Estadística y Computación. Universidad de Cantabria. Email: \{gomezperezai, gomezd\}(AT)unican.es}
\address[2]{ Institut f\"ur Finanzmathematik und Angewandte Zahlentheorie, Johannes Kepler Universit\"at Linz. Email: friedrich.pillichshammer(AT)jku.at }
\cortext[cor1] {Corresponding author}\fntext[fn1]{D. Gómez-Pérez is supported by Consejería De Universidades E Investigación, Medio Ambiente Y Política Social Del Gobierno De Cantabria under Project 21.VP34.64662,``Generación de secuencias para teoría de la comunicación y criptología". This research was started  during a research stay in Johannes Kepler Universit\"at  supported  by the Austrian Science Fund, grant F5509-N26.}
 \fntext[fn2]{F. Pillichshammer is supported by the Austrian Science Fund (FWF) Project F5509-N26, which is a part of the Special Research Program ``Quasi-Monte Carlo Methods: Theory and Applications''.}

\begin{abstract}
The star-discrepancy is a quantitative measure for the irregularity of distribution of a point set in the unit cube that is intimately linked to the integration error of quasi-Monte Carlo algorithms. These popular integration rules are nowadays also applied to very high-dimensional integration problems. Hence multi-dimensional point sets of reasonable size with low discrepancy are badly needed. A seminal result from Heinrich, Novak, Wasilkowski and Wo\'{z}niakowski shows the existence of a positive number $C$ such that for every dimension $d$ there exists an $N$-element point set in $[0,1)^d$ with star-discrepancy of at most $C\sqrt{d/N}$. This is a pure existence result and explicit constructions of such point sets would be very desirable. The proofs are based on random samples of $N$-element point sets which are difficult to realize for practical applications.  

In this paper we propose to use secure pseudorandom bit generators for the generation of point sets with star-discrepancy of order $O(\sqrt{d/N})$. This proposal is supported theoretically and by means of numerical experiments.
\end{abstract}

%%Graphical abstract
%\begin{graphicalabstract}
%\includegraphics{grabs}
%\end{graphicalabstract}

%%Research highlights
%\begin{highlights}
%\item Propose a new construction of low discrepancy point sets for numerical applications based on secure pseudorandom number bit generator.
%\item  Study the relation between the security parameters and the bound on the discrepancy of the generated point set.
%\item  Experimental results to study the constant of bounds for dimensions up to 15.
%\end{highlights}

\begin{keyword}
%% keywords here, in the form: keyword \sep keyword
star-discrepancy\sep inverse of star-discrepancy\sep polynomial tractability\sep pseudorandom numbers
%% PACS codes here, in the form: \PACS code \sep code

%% MSC codes here, in the form: \MSC code \sep code
%% or \MSC[2008] code \sep code (2000 is the default)
 \MSC [2010] 11K3 \sep 11K45
\end{keyword}

\end{frontmatter}

%% \linenumbers

%% main text
\section{Introduction}

Let $d,N \in \NN$. For an $N$-element point set $\cP=\{\bsx_1,\ldots,\bsx_N\}$ in the $d$-dimensional unit-cube $[0,1)^d$ the {\it star-discrepancy} is defined as $$D_N^{\ast}(\cP)=\sup_{\bst \in [0,1]^d} \left|\frac{\#\{k \in \{1,\ldots,N\} \ : \ \bsx_k \in [\bszero,\bst)\}}{N}-{\rm volume}( [\bszero,\bst)) \right|,$$ where, for $\bst=(t_1,\ldots,t_d)\in [0,1]^d$, $[\bszero,\bst)=[0,t_1)\times \ldots \times [0,t_d)$. The star-discrepancy  is a quantitative measure for the irregularity of distribution which is closely related to the integration error of quasi-Monte Carlo (QMC) rules via the well-known Koksma-Hlawka inequality (see \cite{dick2010digital,kuinie,LP14,niesiam}). Traditionally this has found application to different domains such as option pricing, experimental design or antialiasing techniques in computer graphics. 

Classically, the star-discrepancy is studied in an asymptotic setting for a fixed dimension $d$ when $N$ tends to infinity. The best results in this context are point sets with star-discrepancy of order of magnitude $O((\log N)^{d-1}/N)$. Often it is conjectured, that this order of magnitude is the best what one can achieve in dimension $d$. This is a still open question but it must be confessed that there are also other opinions (c.f. \cite{BL13}).

However, in practical applications of QMC methods the dimension $d$ can be very large and in this regime an error bound of order $O((\log N)^{d-1}/N)$ is useless for point sets of manageable size $N$. Consider, for example, a point set with $N=2^d$ elements, which is already very large when $d\ge 25$ (in applications dimension $d$ might be in the hundreds or thousands). Then, $(\log N)^{d-1}/N \approx d^{d-1}/2^d$ which demonstrates how useless the asymptotic excellent discrepancy bounds are for practical problems. Larcher~\cite{L98} was the first to create doubts whether a star-discrepancy bound of  asymptotically excellent order can help for practical applications. 

From late 1990s  this problem has attracted a lot interest, see for example the surveys in  \cite[Section~3.1.5]{NWbook1} and \cite[Section~9.9]{NWbook2} and the references therein. It is now known that there exists some $C>0$ such that for every $N$ and $d$ there exists an $N$-element point set $\cP$ in $[0,1)^d$ whose star-discrepancy satisfies 
\begin{equation}\label{discbd:HNWW}
D_N^{\ast}(\cP) \le C\, \sqrt{\frac{d}{N}}.
\end{equation}
This is a famous result by Heinrich, Novak, Wasilkowski and Wo\'{z}niakowski~\cite{MR1814282}. Aistleitner~\cite{MR2846704} proved that $C$ can be chosen to be 10 and very recently, Gnewuch and Hebbinghaus~\cite{GH18} reduced $C$ further to the value 2.5287\ldots The proofs of these results are based on the fact that the expected star-discrepancy of a randomly generated point set satisfies the inequality in \eqref{discbd:HNWW}. However,  as Aistleitner and Hofer pointed out in \cite{AistHof},  such a mere existence result is not of much use for applications. Furthermore, they proved a version of \eqref{discbd:HNWW} which provides estimates for the probability of a random point set to satisfy it. Their results demonstrate that the probability $q$, that a random point set satisfies a star-discrepancy bound of the form $c(q,d)\, \sqrt{d/N}$, is extremely large already for moderate values of $c(q,d)$ (see the forthcoming Lemma~\ref{le2}). 

A straightforward approach is to simply choose a random sample of $N$ uniformly i.i.d. points in $[0,1)^d$ and then one can assume, with very large probability, that the star-discrepancy of this sample is sufficiently small. This has not being met with general acceptance due to three main shortcomings:
\begin{itemize}
\item First, the computation of the star-discrepancy of a given multi-dimensional point set is very difficult (see, e.g., \cite{Gnewuch2009}). So even for moderate values of $d$ it is almost impossible the check whether the obtained random sample is really of low star-discrepancy. 
\item Second, it is possible to use some heuristic to approximate the discrepancy and search for good candidates. 
There have been studies for searching  for good point sets using known constructions. Doerr and De Rainville \cite{doerr2013constructing} considered generalized Halton sequences and tried to identify good permutations for them. Although the computer experiments are promising, the nature of the work is heuristic.
\item The next problem, which we address in the present work, is how to efficiently generate random samples on a computer from uniformly i.i.d. random variables in $[0,1)^d$.
\end{itemize}

Pseudorandom number generators are any kind of algorithms that generate a sequence of numbers with similar properties to random sequences of numbers. Special classes of pseudorandom number sequences have been proved to be of low star-discrepancy  using the Erd\"os-Turan inequality. Even that this approach is useful for certain applications, i.e to prove normality of specific real numbers (see \cite{bailey2002random}), the obtained discrepancy bounds depend exponentially on the dimension. Our approach is to replace random samples by point sets obtained from secure pseudorandom bit generators. The employed pseudorandom bit generator requires as input a finite string of bits called seed, then the sequence of pseudorandom bits is transformed in a natural way to a point set in the unit-cube $[0,1)^d$ with prescribed  $p$ bits of precision.  Under the assumption of uniformly chosen seed we show that the obtained point sets satisfy the desired discrepancy bound of order $\sqrt{d/N}$ with high probability.  Our proposed method is supported by numerical results up to dimension $d=15$. These results also support a conjecture by Novak and Wo\'{z}niakowski~\cite[p.~68]{NWbook2}. 

We remark that obviously also our approach is not fully constructive. Still we need a random seed for the present construction based on pseudorandom bit generators, even if this is in a way ``less'' random compared for example with the approach in \cite{MR1814282}. 

\medskip

The paper is organized as follows: In Section~\ref{sec:securePRNG} we provide the necessary material related to secure pseudorandom bit generators and the construction of the point sets. Section~\ref{sec:result} is devoted to the theoretical result which guarantees for reasonable values of $d$ and $N$ that the proposed point sets have discrepancy of order $\sqrt{d/N}$ with high probability when the seed is i.i.d. uniformly distributed.  In Section~\ref{sec:numex} we present numerical results which support the proposed use of  secure pseudorandom bit generators for the construction of point sets with low star-discrepancy.

\section{Secure pseudorandom bit generators and point sets in the unit cube}\label{sec:securePRNG}

This section provides the basic notation  that we will employ through the paper.
A \emph{binary word} is a string of \emph{bits}, i.e. of zeros and ones. The set with binary words of length $n$ is denoted by $\{0,1\}^n$. A uniformly distributed random variable on $\{0,1\}^n$ is denoted by $U_n$. We remark that the index $n$ always refers to the length of the binary words. 

\begin{definition}\rm
  Let $m,n \in \NN$, where usually $n$ is much smaller than $m$. A \emph{pseudorandom bit generator} is a function $f:\{0,1\}^n\rightarrow \{0,1\}^m$, which can be evaluated efficiently. Modern implementations can generate up to  113.5 Gigabits per seconds of random bits \cite{Farod}. 

  The {\it seed} or input of a pseudorandom bit generator is a binary word of length $n$. It is assumed that the seed is based on true random information. The output is a pseudorandom bit sequence of length $m$.
\end{definition}

To define a secure pseudorandom bit generator, we use the notion of a $(T,\varepsilon)$-distinguisher (see \cite{farashahi2007efficient}). 

\begin{definition}\rm
Let $T \in \NN$ and $\varepsilon>0$. A \emph{$(T,\varepsilon)$-distinguisher} for a pseudorandom bit generator $f$ is any probabilistic algorithm $\cD$ whose input are binary words of length $m$  and which returns, after at most $T$ time units, a value of 0 or 1 such that
\begin{equation*}
  |\PP[\cD(f(U_n))=1]-\PP[\cD(U_m)=1]|\ge \varepsilon.
\end{equation*}
Then, a pseudorandom number generator $f$ is said to be \emph{$(T,\varepsilon)$-indistinguishable}  
if no $(T,\varepsilon)$-distinguisher exists for $f$.
\end{definition}

A pseudorandom number generator $f$ is said to have $b$ {\it bits of security} if
every $(T,\varepsilon)$-distinguisher for $f$ satisfies $$T> 2^b \varepsilon^2;$$ see \cite{Micciancio2018} for further reading. Equivalently, for a pseudorandom number generator
$f$ with $b$ bits of security, any probabilistic algorithm $\cD$ with run time of at most $T$ time units
that returns a value of 0 or 1 for any input sequence from $\{0,1\}^m$, we have
\begin{equation}
  \label{eq:bound}
  |\PP[\cD(f(U_n))=1]-\PP[\cD(U_m)=1]|< \sqrt{\frac{T}{2^{b}}}\ .
\end{equation}

To initialize a pseudorandom bit generator a seed of $n$ bits must be provided that supports the security strength $b$ requested by the implementation of the generator mechanism. For example,  a Hash-pseudorandom bit generator based on SHA-512 has parameters $n=888$ for $b=256$, see \cite[p.~38, Table~2]{barker2012nist}.
Constructions of pseudorandom number generators that achieve at least $b$ bits security are known for 
any $b$ (see, e.g., \cite{goldreich2003}) but only the cases $b\in \{112, 128, 192, 256\}$ are standardized (see, e.g., \cite{barker2012nist}). We note that this holds for any practical value of $m$.

\paragraph{Transforming binary words to point sets in the unit cube}

To convert a binary word of length $m$ to a $N$-element point set in $[0,1)^d$ we use the following binary interchange format: For a given precision $p \in \NN$, take the $p$ trailing significant field string represented by a binary word $d_{1},d_{2},\ldots,d_p$, then $$\sum_{i=1}^p \frac{d_i}{2^i}$$ is a coordinate of an element of the point set $\cP$. For $p=52$, this representation follows the double-precision floating-point format {\it binary64}, see~\cite{4610935}.

For a given binary word $U_m$ we denote the resulting point set by $\cP(U_m)$, in this case we will require $m=p d N$ bits to construct  $N$ points in $[0,1)^d$.

% In order to construct $N$ points in $[0,1)^d$ we require $m=p d N$ bits.
\section{Theoretical results on star-discrepancy for pseudorandom point sets}\label{sec:result}

We consider a particular distinguisher $\cD$ which takes a binary word $U_m$ of length $m$, generates a sequence of points $\cP=\cP(U_m)=\{X_1, \ldots ,X_N\}$ in the $d$ dimensional unit cube $[0,1)^d$ as described above and calculates the star-discrepancy of $\cP$. The distinguisher returns $1$ if the star-discrepancy is greater than $C \sqrt{d/N}$, where $C$ is a suitably chosen constant. The star-discrepancy of an $N$-element point set in $[0,1)^d$ can be calculated in $O(N^{1+d/2})$ operations (see \cite{Dobkin1996,Doerr2014,Gnewuch2009}).

In order to generate the binary words $U_m$ we use a pseudorandom bit generator
\begin{equation*}
  f:\{0,1\}^n\rightarrow \{0,1\}^m
\end{equation*}
with $b$ bits of security, which will be explicitly selected later. We note that  $n$ is much smaller than $m=p d N$,  in our applications $p=52$, i.e., $m=52 d N$.

\begin{theorem}\label{thm1}
Let $n,p, d, N \in \NN$ and let $m=p d N$. Assume that $f(U_n)$ is obtained from a pseudorandom bit generator $f:\{0,1\}^n\rightarrow \{0,1\}^m$ with $b$ bits of security. Let $\cP(f(U_n))$ be the corresponding $N$-element point set in $[0,1)^d$ with random seed $U_n$. Then we have
\begin{eqnarray*}
\lefteqn{\PP  \left [D_N^{\ast}(\cP(f(U_n))) \ge  C \, \sqrt{\frac{d}{N}} \right ]}\\
& \le & \exp\left(4.9- \frac{1}{5.7^2}\left(C- \frac{\sqrt{d N}}{2^p}\right)^2\right) + O\left(\sqrt{\frac{N^{1+d/2}}{2^{b}}}\right).
\end{eqnarray*}
\end{theorem}

 For a choice of values of bit security $b=256$ and precision $p=52$, the bound on the probability in Theorem~\ref{thm1} is very small for reasonable values of $d$ and $N$ whenever $C>\sqrt{4.9} \cdot 5.7 =12.6174\ldots$. To be more precise, with these values for $p$ and $b$ the above bound on the probability is essentially $\exp(4.9-(C/5.7)^2)=:g(C)$ as long as $N^{1+d/2} \ll 2^{256}$. We have the following (rounded) values:
 \renewcommand{\arraystretch}{1.5}
$$\begin{array}{r||c|c|c|c|c|c}
C & 15 & 16 & 17 & 18 & 19 & 20\\
\hline
g(C) & 1.3\cdot 10^{-1}  & 5.1 \cdot 10^{-2}& 1.8 \cdot 10^{-2}& 6.3 \cdot 10^{-3}& 2.0 \cdot 10^{-3}& 6.0 \cdot 10^{-4} 
\end{array} 
$$

For the proof of Theorem~\ref{thm1} we need the following lemmas.

\iffalse
$$
\PP[\cD(f(U_n))=1] < \delta,
$$
or equivalently,
$$
\PP  \left [D_N^{\ast}(\cP(f(U_n))) \ge C \, \sqrt{\frac{d}{N}} \right ]< \delta.
$$

We have 
\begin{eqnarray*}
\lefteqn{\PP  \left [D_N^{\ast}(\cP(f(U_n))) \ge C \, \sqrt{\frac{d}{N}} \right ]}\\
& \le & \PP  \left [D_N^{\ast}(\cP(U_{52dN})) \ge C \, \sqrt{\frac{d}{N}} \right ] \\
& & +
\left| \PP  \left [ D_N^{\ast}(\cP(f(U_n))) \ge C \, \sqrt{\frac{d}{N}} \right ] -
\PP  \left [D_N^{\ast}(\cP(U_{52dN})) \ge C \, \sqrt{\frac{d}{N}} \right ] \right|. 
\end{eqnarray*}

The first term can be bound using Markov inequality, i.e.
$$
\left| \PP  \left [D_N^{\ast}(\cP(U_{52dN})) \ge C \, \sqrt{\frac{d}{N}} \right] \right|\le 
\frac{\EE\left [D_N^{\ast}(\cP(U_{52dN})) \right ] }{C \, \sqrt{d/N}} = 
\frac{10}{C}.
$$
The second term can be bounded using equation \eqref{eq:bound} and the fact that 
to calculate the discrepancy can be done in $N^{d/2}$ operations \cite{Gnewuch2009},
\begin{equation*}
\left| \PP  \left [D_N^{\ast}(\cP(f(U_n))) \ge C \, \sqrt{\frac{d}{N}} \right ] -
\PP  \left [D_N^{\ast}(\cP(U_{52dN})) \ge C \, \sqrt{\frac{d}{N}} \right ]  \right|<
\sqrt{\frac{N^{d/2}}{2^{b}}}.
\end{equation*}

Taking $b = \left\lceil 2 \, \ld C + \tfrac{d}{2} \, \ld N \right\rceil$  we obtain  
\begin{equation*}
\left| \PP  \left [D_N^{\ast}(\cP(f(U_n))) \ge C \, \sqrt{\frac{d}{N}} \right ] -
\PP  \left [D_N^{\ast}(\cP(U_{52dN})) \ge C \, \sqrt{\frac{d}{N}} \right ]  \right|< \frac{1}{C} .
\end{equation*}
so $\delta= 11/C$.
\fi

\begin{lemma}\label{le1}
Let $X_1,\ldots,X_N$ be uniformly i.i.d. in $[0,1)^d$ with $X_j=(X_{j,1},\ldots,X_{j,d})$ for $j \in \{1,\ldots,N\}$. For $p \in \NN$ and $i \in \{1,\ldots,d\}$ let $$X_{j,i}^{(p)}:= \frac{\lfloor 2^p X_{j,i} \rfloor}{2^p} \in \left\{0,\frac{1}{2^p},\ldots,\frac{2^p -1}{2^p}\right\}\ \ \mbox{ and }\ \  X_j^{(p)}=(X_{j,1}^{(p)},\ldots,X_{j,d}^{(p)}).$$ Consider the binary expansion of $X_{j,i}^{(p)}$ and concatenate the corresponding $d N$ binary words to obtain a binary word $U_m$ of length $m=p d N$. Then 
\begin{enumerate}
\item $\cP(U_m)=\{X_1^{(p)},\ldots,X_N^{(p)}\}$;
\item $U_m$ is uniformly distributed in $\{0,1\}^m$;
\item the star-discrepancies of $\cP=\{X_1,\ldots,X_N\}$ and $\cP(U_m)$ differ at most by $d/2^p$, i.e., 
\begin{equation}\label{diffdisc}
|D_N^{\ast}(\cP)-D_N^{\ast}(\cP(U_m))| \le \frac{d}{2^p};
\end{equation}
\item for $C>0$ we have  \begin{equation}\label{contDisc}
\PP\left[ D_N^{\ast}(\cP(U_m)) \ge C \, \sqrt{\frac{d}{N}}\right] \le \PP\left[ D_N^{\ast}(\cP) \ge C \, \sqrt{\frac{d}{N}} - \frac{d}{2^p}\right].
\end{equation}
\end{enumerate}
\end{lemma}

\begin{proof}
Items {\it 1.} and {\it 2.} are obviously true. In order to prove {\it 3.} note first that for every $j \in \{1,\ldots,N\}$ we have $$\|X_j-X_j^{(p)}\|_{\infty} \le \frac{1}{2^p},$$ where $\|\cdot\|_{\infty}$ denotes the $\ell_{\infty}$-norm in $\RR^d$. Therefore \eqref{diffdisc} follows from \cite[Proposition~3.15]{dick2010digital}. Finally we show {\it 4.} For $C>0$ we obtain from \eqref{diffdisc} that $$D_N^{\ast}(\cP(U_m)) \ge C \, \sqrt{\frac{d}{N}} \ \Rightarrow \ D_N^{\ast}(\cP) \ge C \, \sqrt{\frac{d}{N}} - \frac{d}{2^p}$$ and hence we obtain \eqref{contDisc}.
\end{proof}

The following result is \cite[Corollary~1]{AistHof}.

\begin{lemma}[Aistleitner and Hofer]\label{le2}
For any $d, N \in \NN$ and $q \in (0,1)$ a randomly generated $d$-dimensional point set $\cP$ in $[0,1)^d$ satisfies 
\begin{equation*}
D_N^{\ast}(\cP) \le 5.7 \sqrt{4.9+\log((1-q)^{-1})} \, \sqrt{\frac{d}{N}}
\end{equation*}
with probability at least $q$.
\end{lemma}

Now we can state the proof of Theorem~\ref{thm1}.

\begin{proof}[Proof of Theorem \ref{thm1}]
Let $\cP$ be a set of uniformly i.i.d. points in $[0,1)^d$ and let $\cP(U_m)$ be constructed as in the statement of Lemma~\ref{le1}. Then we have
\begin{eqnarray*}
\lefteqn{\PP  \left [D_N^{\ast}(\cP(f(U_n))) \ge C \, \sqrt{\frac{d}{N}} \right ]}\\
& \le & \PP  \left [D_N^{\ast}(\cP(U_m)) \ge C \, \sqrt{\frac{d}{N}} \right ] \\
& & +
\left| \PP  \left [ D_N^{\ast}(\cP(f(U_n))) \ge C \, \sqrt{\frac{d}{N}} \right ] -
\PP  \left [D_N^{\ast}(\cP(U_m)) \ge C \, \sqrt{\frac{d}{N}} \right ] \right|. 
\end{eqnarray*}

The first term can be bounded using Lemma~\ref{le1} and Lemma~\ref{le2}. In particular, choosing $q$ such that $$5.7 \sqrt{4.9+\log((1-q)^{-1})} \, \sqrt{\frac{d}{N}}=C \, \sqrt{\frac{d}{N}} - \frac{d}{2^p}$$ 
we get
\begin{equation}\label{bdAH}
\PP\left[ D_N^{\ast}(\cP(U_m)) \ge C \, \sqrt{\frac{d}{N}}\right] \le {\rm e}^{4.9- \frac{1}{5.7^2}\left(C- \frac{\sqrt{d N}}{2^p}\right)^2}.
\end{equation}

The second term can be bounded using Eq.~\eqref{eq:bound} and the fact that 
the star-discrepancy can be calculated in $O(N^{1+d/2})$ operations (see \cite{Dobkin1996,Doerr2014}). This way we obtain  \begin{equation*}
\left| \PP  \left [D_N^{\ast}(\cP(f(U_n))) \ge C \, \sqrt{\frac{d}{N}} \right ] -
\PP  \left [D_N^{\ast}(\cP(U_m)) \ge C \, \sqrt{\frac{d}{N}} \right ]  \right|<
O\left(\sqrt{\frac{N^{1+d/2}}{2^{b}}}\right).
\end{equation*}

Together we obtain 
$$\PP  \left [D_N^{\ast}(\cP(f(U_n))) \ge C \, \sqrt{\frac{d}{N}} \right ] \le {\rm e}^{4.9- \frac{1}{5.7^2}\left(C- \frac{\sqrt{d N}}{2^p}\right)^2} + O\left(\sqrt{\frac{N^{1+d/2}}{2^{b}}}\right).$$
\end{proof}

Denoting $\ld$  the binary logarithm and choosing $b = \left\lceil 2 \, \ld C + \tfrac{d}{2} \, \ld N \right\rceil$  in Theorem~\ref{thm1} we obtain  the following corollary:

\begin{corollary}\label{co1}
With the notation from Theorem~\ref{thm1} and with $b=\left\lceil 2 \, \ld C + \tfrac{d}{2} \, \ld N \right\rceil$ we obtain  $$\PP  \left [D_N^{\ast}(\cP(f(U_n))) \ge C \, \sqrt{\frac{d}{N}} \right ] \le O\left(\frac{1}{C^2}\right).$$ 
\end{corollary}

\section{Numerical experiments}\label{sec:numex}
In order to support our approach, we have performed several numerical experiments during 83 days in one node (332.8 GFLOPS) of  Altamira, the High Performance Computer cluster at the University of Cantabria. Altamira belongs to RES (Red Española de Supercomputación) network.

We generated a file with 40.000 hexadecimal digits with the library Openssl \cite{openssl}, using the default AES-256 CTR\_DRBG with security bit strength $b=256$. In this configuration, the required seed length is $n=384$ bits in which at least 256 bits are supposed to be random and independent. The implementation employs entropy sources available to the computer (keystrokes, mouse movements, internal clock,...).

These hexadecimal bits are transformed into floating point numbers of precision $p=52$ and
then, depending of the dimension, grouped in $d$-dimensional
vectors following the steps described in Section~\ref{sec:securePRNG}.
Finally the star-discrepancy is computed using the implementation of the algorithm by Dobkin, Eppstein and Mitchell \cite{Dobkin1996} provided by Magnus Wahlström~\cite{Charles2013}.

First we present a comparison between the points generated and a Halton sequence for dimension $d=10$, see Figure~\ref{fig:Dim10}. The discrepancy of the generated points (plotted in blue)  is always smaller than the one corresponding to the Halton sequence (plotted in green).
\begin{figure}
\begin{center}
\includegraphics[width=13cm]{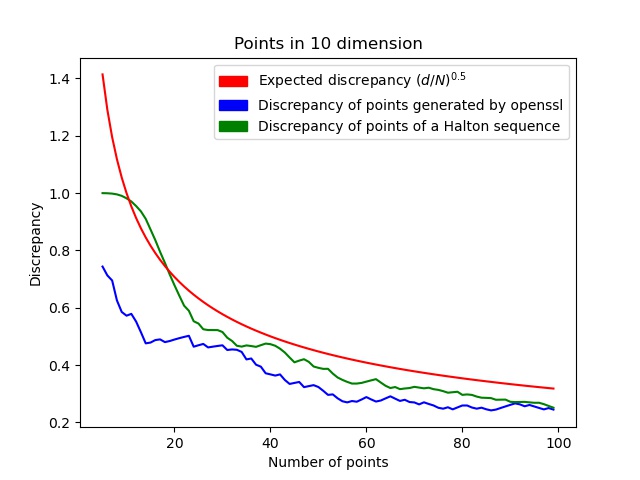}
\end{center}
\caption{Discrepancy of $N$ points in  $[0,1)^{10}$, comparing with a Halton sequence where the bases are the first $10$ primes.}
\label{fig:Dim10}
\end{figure}

Further we present more results for the proposed point sequences in dimensions $d \in \{13,14,15\}$, which are displayed in  Figures~\ref{fig:Dim13} - \ref{fig:Dim15}. The real values plotted in blue are compared with the graph of the bound $\sqrt{d/N}$ plotted in red. 

\begin{figure}
\begin{center}
\includegraphics[width=13cm]{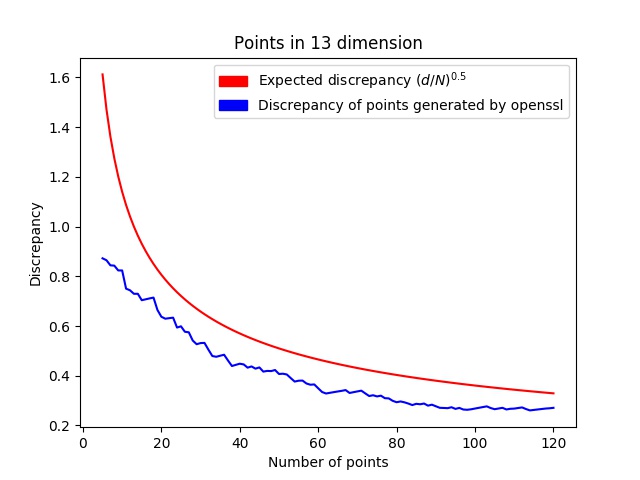}
\end{center}
\caption{Discrepancy of $N$ points in  $[0,1)^{13}$}
\label{fig:Dim13}
\end{figure}

\begin{figure}[h]
\begin{center}
\includegraphics[width=13cm]{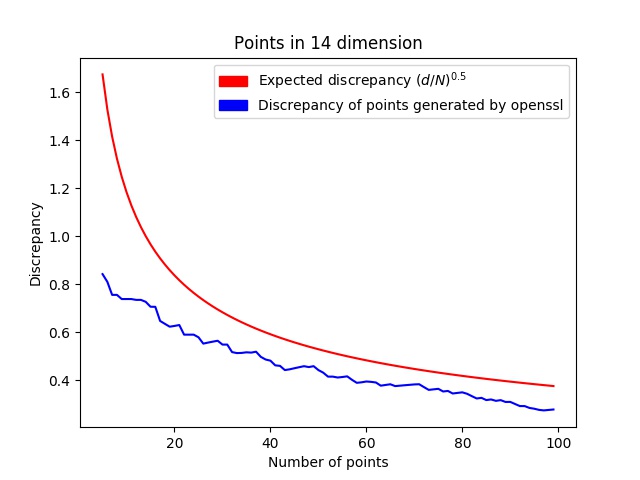}
\end{center}
\caption{Discrepancy of $N$ points in  $[0,1)^{14}$}
\label{fig:Dim14}
\end{figure}

\begin{figure}[h]
\begin{center}
\includegraphics[width=13cm]{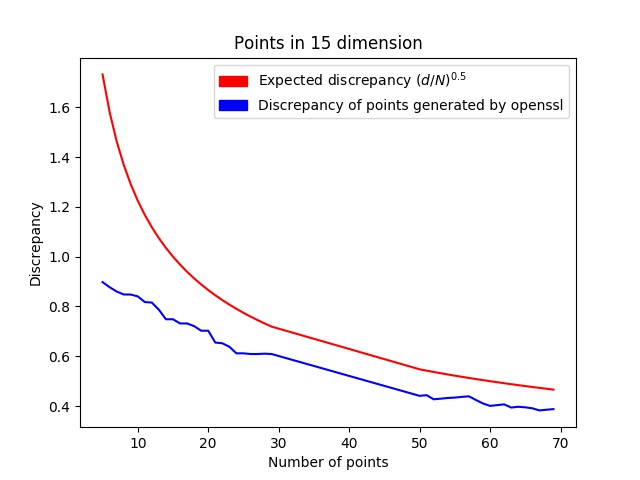}
\end{center}
\caption{Discrepancy of $N$ points in  $[0,1)^{15}$}
\label{fig:Dim15}
\end{figure}

We are also interested in the {\it inverse of star-discrepancy}, which is for given $\varepsilon \in (0,1)$  and dimension $d$ the smallest size $N$ of a point set in $[0,1)^d$ with star-discrepancy at most $\varepsilon$, i.e., $$N(\varepsilon,d)=\min \{N \in \NN \ : \ \exists \cP \subseteq [0,1)^d \ \mbox{ such that $|\cP|=N$ and $D_N^*(\cP) \le \varepsilon$}\}.$$ Note that \eqref{discbd:HNWW} implies that $$N(\varepsilon,d) \le C^2 d \varepsilon^{-2}.$$ This means in the language of IBC (information based complexity) that the star-discrepancy is {\it polynomially tractable}.

Results on $N(\varepsilon,d)$ based on the proposed construction in this paper for $d\in \{2,3,\ldots,15\}$ and $\varepsilon \in \{0.5,0.333,0.25,0.125\}$ are displayed in Table~\ref{tab:1}. Based on results in \cite{DGW10} for $d<10$ it is conjectured by Novak and Wo\'{z}niakowski in \cite[p.~68]{NWbook2} that $N=10 d$ points in $[0,1)^d$ suffice in order to achieve a star-discrepancy of at most $0.25$. Our values $N(0.25,d)$ from Table~\ref{tab:1} are in this conjectured regime.

\begin{table}[hbt!]
 \centering
 {\scriptsize
\begin{tabular}{|c||c|c|c|c|c|c|c|c|c|c|c|c|c|c||c|}
\hline
$d$  & 2 & 3 &4 & 5 & 6 &7 & 8 & 9 &10&  11& 12 &13 &14&15&$\varepsilon$\\
\hline
$N(\varepsilon,d)$ & 10&13&10&25&26&24&28&31&14&30&32&33&38&50& 0.5\\
\hline
$N(\varepsilon,d)$ & 10&45&36&55&60&52&62&61&47&70&75&62&83&-& 0.333\\
\hline
$N(\varepsilon,d)$ & 11&87&78&93&82&76&96&94&76&108&120&-&-&-& 0.25\\
\hline
$N(\varepsilon,d)$ & 215&195&194&232&214&325&304&-&-&-&-&-&-&-& 0.125\\
\hline
\end{tabular}
}
\caption{Dimension $d$ versus $N(\varepsilon,d)$, the number of points generated using AES-256 CTR\_DRBG necessary to reach a discrepancy of at most $\varepsilon$, for $\varepsilon \in \{0.125, 0.25, 0.333, 0.5\}$}
\label{tab:1}
\end{table}

Finally, in Figure~\ref{ptsets2D} we visualize two point sets in the two-dimensional cube. The one on the left hand side consists of 11 points and has star-discrepancy not larger that $0.25$. The point set on the right hand side consists of 242 points and has star-discrepancy not larger than $0.1$. 

\begin{figure}[h!]
     \centering
     \subfigure{\includegraphics[height=60mm]{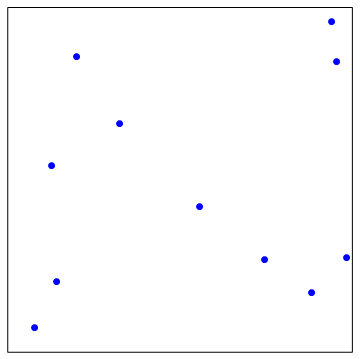}}
     \hspace{.1in}
     \subfigure{\includegraphics[height=60mm]{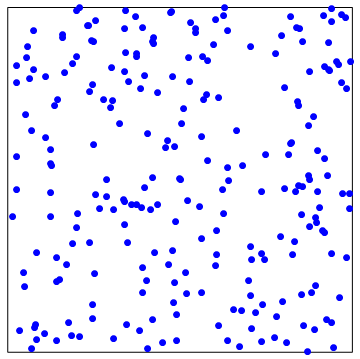}} 
     \caption{Point sets in $[0,1)^2$ with star-discrepancy less than 0.25 (11 points, left picture) and less than 0.1 (242 points, right picture) }
     \label{ptsets2D}
\end{figure}

\section{Conclusions}

In this work, we employed secure pseudorandom bit generators for the generation of point sets in $[0,1)^d$ of reasonable size with low star-discrepancy. The expected probability of obtaining point sets with star-discrepancy of order $C \sqrt{d/N}$ for specific parameters were given in Theorem~\ref{thm1}. Previous theoretical results were pessimistic with respect to the value of  the constant $C$ whereas our numerical results suggest that the practical values are closer to 1.  We hope that this study provides new insight into the problem of generating low discrepancy multidimensional point sets.

\section*{Acknowledgement}
The authors want to thank Carola Doerr for her remarks during the preparation of the article as well as two anonymous referees.

\end{document}